\documentclass[a4paper,12pt]{article}
\usepackage[utf8]{inputenc}
\usepackage[lmargin=3cm,tmargin=3cm,rmargin=2cm,bmargin=2cm]{geometry}
\usepackage[onehalfspacing]{setspace}
\usepackage[T1]{fontenc}
\usepackage[english]{babel}
\usepackage{fourier}
\usepackage{mathtools}
\usepackage{bm}
\usepackage{esvect}
\usepackage{graphicx,xcolor,comment,enumerate,multirow,multicol,indentfirst,tabu,longtable}
\usepackage{amsmath,amsthm,amsfonts,amssymb,dsfont,mathtools,blindtext}
\usepackage{graphicx,subfigure}
\usepackage{tcolorbox}

\usepackage{hyperref}
\hypersetup{
    colorlinks=true,
    linkcolor=blue,
    filecolor=blue,      
    urlcolor=blue,}

\newtheorem{theorem}{Theorem}[section]

\newtheorem{lemma}{Lemma}[section]
\newtheorem{corollary}{Corollary}[section]

\newcommand{\R}{\mathds{R}}

\title{Analytic conjugation between planar differential systems 
 and potential systems}
\author{F.\,J.\,S.\, NASCIMENTO\footnote{ Geology Collegiate, Federal University of Vale do São Francisco, Senhor do Bonfim, Bhaia, Brazil\\
 Email: \href{mailto:francisco.jsn@univasf.edu.br?subject=&body=}{francisco.jsn@univasf.edu.br} \\
 ORCID:\href{https://orcid.org/0009-0004-9536-7979}{0009-0004-9536-7979}}}
\date{}
\begin{document}
\pagestyle{plain}
\maketitle
\begin{abstract}
   In this article it is proved that an analytical planar vector field with a non-degenerate center at $(0,0)$ is analytically conjugate, in a neighborhood of $(0,0)$, to a Hamiltonian vector field of the form $y\frac{\partial}{\partial x}-V'(x)\frac{\partial}{\partial y}$, where $V$ is an analytic function defined in a neighborhood of the origin such that $V(0)=V'(0)=0$ and $V''(0)>0.$
\end{abstract}

\section{Introduction}
Let $\Omega$ be an open subset of $\R^2$ and $C^\omega(\Omega,\R^d)$ the set of the real analytic functions defined on $\Omega$ with values on $\R^d$, $d\in\{1,2\}$. Let $P,Q\in C^\omega(\Omega,\R)$ and consider the analytic differential system
\begin{equation}\label{eq1}
    \Dot{x}=P(x,y),\quad \Dot{y}=Q(x,y), \quad (x,y)\in \Omega.
\end{equation}
The system (\ref{eq1}) define on $\Omega$ the planar vector field $X=(P,Q)\in C^\omega(\Omega,\R^2)$. In this paper the vector field $X=(P,Q)$ it will often be represented by the differential operator  
\begin{equation}\label{eq2}
    X=P\frac{\partial}{\partial x}+Q\frac{\partial}{\partial y}.
\end{equation}
A point $p\in \Omega$ such that $X(p)=(0,0)$ is called a singular
point of $X$. A singular point $p$ is non-degenerate if
the determinant of Jacobian matrix $DX(p)$ of $X$ at $p$ is non-zero, that is, if $P_x(p) Q_y(p)- P_ y(p) Q_ x(p)\neq 0$. A singular point $p$ is called a center of $X$ if there exists an open
neighborhood $U$ of $p$ such that each solution of (\ref{eq1}) with initial condition on $U-\{p\}$ defines a
periodic orbit surrounding $p$. The largest neighborhood $\mathcal{A}$ with this property is called the period annulus of $p$. Let $p$ be a center of $X$ and let $T(q)$ denote the period of orbit passing through $q\in \mathcal{A}$. The function $q\to T(q)$ is called the period function associated with the center $p$.

Let $\Omega_1$ and $\Omega_2$ be open sets of $\R^2$. The vector fields $X\in C^\omega(\Omega_1,\R^2)$ and $Y\in C^\omega(\Omega_2,\R^2)$ are analytically equivalent (or analytically conjugacy) if exists an analytic diffeomorphism $h:\Omega_1\to\Omega_2$ such that 
\begin{equation}\label{eq3}
    D_qhX(q)=Y(h(p))\quad \text{for every }\quad q\in \Omega_1.
\end{equation}
The  diffeomorphism $h$ maps singular points to singular points, and periodic orbits to periodic orbits, preserved  the period of the periodic orbits. Let $p$ be a singular point $p$ of $X$, we say that $X$ is locally analytically conjugate to a vector field $Y$ if equality (\ref{eq3}) holds on a neighborhood of $p$. In this paper, we shall show the following 
\begin{theorem}\label{teo1}
  Let $\Omega$ be an open subset of \, $\R^2$ such that $(0,0)\in \Omega$ and suppose that the vector field $X\in C^\omega(\Omega,\R^2)$ has a non-degenerate center at $(0,0)$. Then $X$ is analytically conjugate, on a neighborhood of $(0,0)$, to the vector field 
  \begin{equation}\label{eq4}
Y=y\frac{\partial}{\partial x}-V'(x)\frac{\partial}{\partial y}, \end{equation} where $V$ is an analytic function defined in a neighborhood of the origin such that $V(0)=V'(0)=0$ and $V''(0)>0.$
\end{theorem}

 The vector field (\ref{eq4}) defines the Hamiltonian system 
 \begin{equation}\label{eq5}
\Dot{x} =y, \quad\Dot{y}=-V'(x).
\end{equation}  This system is called a potential system and has been the subject of study by several researchers. Among the issues addressed, two stand out, and both are related to the period function associated with the center of system (\ref{eq5}). The first question is dedicated to the study of the monotonicity of the period function. This issue is addressed in several articles, for example, \cite{chicone1987monotonicity}, \cite{gasull1997period}, \cite{chavarriga1999survey}, \cite{sfecci2015isochronous}. The second question concerns the possibility of constructing a potential function V from a positive function T. This question is called the inverse problem and has been addressed in several articles, for example,  \cite{urabe1964relations}, \cite{urabe1961potential}, \cite{alfawicka1984inverse}, \cite{alfawicka84inverse}, \cite{manosas2008two},  \cite{kamimura2010global}, \cite{kamimura2013global}. In \cite{chicone1987monotonicity} Chicone questions under what conditions a vector field $X$ with a center at $(0,0)$ is conjugate to a Hamiltonian vector field of type (\ref{eq4}). The theorem (\ref{teo1}) is an answer to the question proposed by Chicone for the particular case in which $X$ is analytic and the center at $(0,0)$ is non-degenerate.
 
 \section{Proof of Theorem \ref{teo1}}

The central idea of the proof consists of constructing a vector field of the type (\ref{eq4}), from the period function of the vector field $X$, in such a way that the period function of $Y$ is equal to the period function of $X$. The proof will be divided into some lemmas.

 \begin{lemma}\label{lem1}
     Let $X=P\frac{\partial}{\partial u}+Q\frac{\partial}{\partial v}$ be an analytical vector field with a non-degenerate center at $(0,0)$. Then $X$ is analytically conjugate, on a neighborhood of $(0,0)$, to the vector field 
  \begin{equation}\label{eq6}
X(\xi,\eta)=f(\xi^2+\eta^2)\bigg(\eta\frac{\partial}{\partial \xi}-\xi\frac{\partial}{\partial \eta}\bigg), \end{equation} where $f$ is an analytic function defined in a neighborhood of the origin such that $f(0)>0.$
 \end{lemma}
\begin{proof}
    It is an immediate consequence of the Poincaré Normal Form Theorem (see \cite{mardesic1995linearization}).
\end{proof}

Note that (\ref{eq6}) is a Hamiltonian vector field with Hamiltonian function $H(\xi,\eta)=F(\xi^2+\eta^2)$, where $F$ is the analytic function defined by
\begin{equation}\label{eq7}
    F(z)=\frac{1}{2}\int_0^zf(s)ds.
\end{equation}
Therefore, the periodic orbits of (\ref{eq6}) are contained in the level curves $H(\xi,\eta)=E.$ 
\begin{lemma}\label{lem2}
    Let $E\geq0$, the period function $T(E)$ of (\ref{eq6}) parameterized by $H(\xi,\eta)=E$ is the analytic function given by 
    \begin{equation}\label{eq08}
        T(E)=\frac{\pi}{F'(F^{-1}(E))}=\pi\frac{d}{dE}F^{-1}(E).
    \end{equation}
\end{lemma}
\begin{proof}
    In polar coordinates (\ref{eq6}) becomes  \begin{equation}\label{eq8}
X(r,\theta)=f(r^2)\frac{\partial}{\partial \theta},
\end{equation} where $r^2=\xi^2+\eta^2$ and 
 $\theta=\arctan\big(\frac{\eta}{\xi}\big)$. Thus, the origin is a center with the periodic orbits inside the circles $\xi^2+\eta^2=r^2.$ Therefore the period of the periodic orbit of (\ref{eq6}) inside the circle of radius $r$ is given by 
  \begin{equation}\label{eq10}
        \Hat{T}(r)=\frac{2\pi}{f(r^2)}.
    \end{equation}
    Let $E>0$ be such that $H(\xi,\eta)=F(\xi^2+\eta^2)=E$. Since $F'(0)=f(0)>0$ we have that $F$ has an analytic inverse in a neighborhood of zero. Therefore, $r^2=\xi^2+\eta^2=F^{-1}(E)$ in a neighborhood of zero. Substituting $r=\sqrt{F^{-1}(E)}$ in (\ref{eq10}) we get
    \begin{equation}\label{eq11}
       T(E)=\Hat{T}(\sqrt{F^{-1}(E)})=\frac{2\pi}{f(F^{-1}(E))}\overset{(\ref{eq7})}{=}\frac{\pi}{F'(F^{-1}(E))}.
    \end{equation}
    Note that, by definition, $T(E)$ is analytic in a neighborhood of zero with $T(0)=\pi/F'(0)>0$.
\end{proof}
\begin{lemma}\label{lem3}
 Let $T(E)$ be the analytic function definite in (\ref{eq11}). Then exist an analytic function $V$, defined in a neighborhood of zero, such that $V(0)=V'(0)=0$ and $V''(0)>0$. Furthermore, the period function of the Hamiltonian vector field, defined in a neighborhood of $(0,0)$, by \begin{equation}\label{eq12}
Y(x,y)=y\frac{\partial}{\partial x}-V'(x)\frac{\partial}{\partial y}, \end{equation} is equal to $T(E)$.
\end{lemma}
\begin{proof}
    Since $F^{-1}(E)$ is analytic in the neighborhood of zero, with $F^{-1}(0)=0$, there exists a sequence of real numbers $(a_n)_{n \geq1}$ such that
    \begin{eqnarray}
    F^{-1}(E)=\sum_{n\geq1}a_nE^n.
\end{eqnarray} Therefore,
\begin{eqnarray}\label{eq14}
\big(F^{-1}\big)'(E)=\sum_{n\geq1}na_nE^{n-1}.
\end{eqnarray} Since $\big(F^{-1}\big)'(0)=1/F'(0)$, we have that $a_1=1/F'(0)>0.$ Let $(b_n)_{n\geq1})$ be the sequence of real numbers defined by 
\begin{eqnarray}
b_n=\frac{n\sqrt{2\pi}\Gamma(n)}{4\Gamma(n+1/2)}a_n,
\end{eqnarray} where $\Gamma$ is Euler's gamma function. Since $$\lim\limits_{n\to\infty}\frac{\Gamma(n)}{\Gamma(n+1/2)}=0,$$ there is a constant $M>0$ such that $|b_n|\leq nM|a_n|.$ Therefore, the function $\varphi(E)$ defined by \begin{eqnarray}\varphi(z)=\sum_{n\geq1}b_nz^{2n-1}\end{eqnarray} is analytic in a neighborhood of zero , with $\varphi(0)=0$ and $\varphi'(0)=b_1=\sqrt{2}a_1/2>0.$ Let $\zeta(E)$ be defined in a neighborhood of zero by
\begin{eqnarray}\label{zeta}
\zeta(E)=\varphi(\sqrt{E})=\sum_{n\geq1}b_nE^{(2n-1)/2}.
\end{eqnarray} Since $$\zeta'(E)=\frac{\varphi'(E)}{2\sqrt{E}},$$ it follows that $\lim\limits_{E\to 0^+} \zeta'( E)=+\infty$ and, consequently, $\zeta(E)$ is invertible in a neighborhood of zero. Let $V$ be the function defined in a neighborhood of zero by $V(\tilde{x})=\zeta^{-1}(\tilde{x})$. By definition, $V$ is analytic in a neighborhood of zero with $V(0)=V'(0)=0$ and $V''(0)>0$. In fact, since $\varphi'(0)>0$, $\varphi(z)$ is invertible in a neighborhood of zero. Therefore, $x=\zeta(E)=\varphi(\sqrt{E})$ implies that $\zeta^{-1}(x)=[\varphi^{-1 }(x)]^2$. So $V(x)=\zeta^{-1}(x)=[\varphi^{-1}(x)]^2$ is analytic in a neighborhood from zero. Furthermore, it follows from the definition of $V$ that $V(0)=V'(0)=0$ and $$V''(0)=2(\varphi^{-1})'(0)= 2/\varphi'(0)=2/b_1>0.$$
Note that, by definition, $V$ is an even function. Let $Y$ be the Hamiltonian vector field defined by \begin{equation}\label{eq18}
Y(x,y)=y\frac{\partial}{\partial x}-V'(x)\frac{\partial}{\partial y}. 
\end{equation}  Note that, by definition, $Y$ is analytic in a neighborhood of $(0,0)$ and has a non-degenerate center at $(0,0)$. Therefore, the periodic orbits of (\ref{eq18}) are contained in the level curves $H(x,y)=E,$ where $H$ is the Hamiltonian function defined in a neighborhood of $(0,0)$ by 
\begin{eqnarray}\label{eq19}
    H(x,y)=\frac{y^2}{2}+V(x).
\end{eqnarray}

Let $\hat{T}(E)$  be the period of the periodic orbit with $H(x,y) = E$ of the vector field (\ref{eq18}). By (\ref{eq19}) we have to 
\begin{eqnarray}
    y=\pm \sqrt{2(E-V(x))}.
\end{eqnarray}
Then 
\begin{eqnarray*}
    \hat{T}&=&2\int_{V_-^{-1}(E)}^{V_+^{-1}(E)}\frac{dx}{y}=2\int_{V_-^{-1}(E)}^{V_+^{-1}(E)}\frac{dx}{\sqrt{2(E-V(x))}}\\
    &=&\sqrt{2}\int_{0}^{V_+^{-1}(E)}\frac{dx}{\sqrt{E-V(x)}}-\sqrt{2}\int_{0}^{V_-^{-1}(E)}\frac{dx}{\sqrt{E-V(x)}},
\end{eqnarray*} where $V_-^{-1}$ and $V_+^{-1}$ denote the inverse of $V$ at $x<0$ and $x>0$ respectively. The change of coordinates $x=V_+^{-1}(E)$ and $x=V_-^{-1}(E)$ in the first and second integral above respectively yield
\begin{eqnarray}
   \hat{T}(E)=\sqrt{2}\int_{0}^{E}\frac{\big(V_+^{-1}(s)-V_-^{-1}(s)\big)'ds}{\sqrt{E-s}}.
\end{eqnarray} Since $V$ is even, we have that $V_-^{-1}(E)=-V_+^{-1}(E)$ and, consequently, $V_+^{-1}(E)-V_-^{-1}(E)=2V_+^{-1}(E)=2V^{-1}(E)$. Therefore, the period function $\hat{T}(E)$ of (\ref{eq18}) parameterized by the energy levels $H=E$ is the analytical function defined by
\begin{equation}\label{eq20}
\Hat{T}(E)=2\sqrt{2}\int_{0}^{E}\frac{(V^{-1})'(s)ds}{\sqrt{E-s}}.
 \end{equation} 
 By definition, $$V^{-1}(E)=\varphi(\sqrt{E})\overset{(\ref{zeta})}{=}\sum_{n\geq1}b_nE^{(2n-1)/2}$$  and, therefore, $$(V^{-1})'(E)=\sum_{n\geq2}\frac{(2n-1)}{2}b_nE^{(2n-3)/2}.$$ Substituting the series of $(V^{-1})'(E)$ into (\ref{eq20}), we obtain

 \begin{eqnarray*}
	\Hat{T}(E)&=&2\sqrt{2}\sum_{n\geq2}\frac{(2n-1)}{2}b_n\int_{0}^{E}\frac{s^{(2n-3)/2} ds}{\sqrt{(E-s)}}\\
&=&2\sqrt{2}\sum_{n\geq2}\frac{(2n-1)}{2}b_n\int_{0}^{1}\frac{(Et)^{(2n-3)/2}E dt}{\sqrt{(E-Et)}}\\
	&=&2\sqrt{2}\sum_{n\geq2}\frac{(2n-1)}{2}b_nE^{n-1}\int_{0}^{1}\frac{t^{(2n-3)/2}dt}{\sqrt{(1-t)}}\\
 &=&2\sqrt{2}\sum_{n\geq1}\frac{(2n-1)}{2}b_nE^{n-1}2\frac{\sqrt{\pi}\Gamma(n-1/2)}{\Gamma(n)}\\
  &=&2\sqrt{2}\sum_{n\geq1}b_n\frac{\sqrt{\pi}\Gamma(n+1/2)}{\Gamma(n)}E^{n-1}.\\
\end{eqnarray*} 
By definition, $b_n=\frac{n\sqrt{2\pi}\Gamma(n)}{4\Gamma(n+1/2)}a_n.$ Then $$b_n\frac{\sqrt{\pi}\Gamma(n+1/2)}{\Gamma(n)}E^{n-1}=\frac{n\sqrt{2\pi}\Gamma(n)}{4\Gamma(n+1/2)}a_n\cdot\frac{\sqrt{\pi}\Gamma(n+1/2)}{\Gamma(n)}=\frac{\sqrt{2}\pi na_n}{4}.$$ Therefore 
$$\Hat{T}(E)=2\sqrt{2}\sum_{n\geq1}\frac{\sqrt{2}\pi na_n}{4}E^{n-1}=\pi\sum_{n\geq1}na_nE^{n-1}\overset{(\ref{eq14})}{=}\pi(F^{-1})'(E)=\frac{\pi}{F'(F^{-1}(E))}\overset{(\ref{eq08})}{=}T(E).$$
\end{proof}
\subsection{Conclusion of the proof of Theorem \ref{teo1}.}
By construction, the vector field (\ref{eq18}) has a non-degenerate center at $(0,0)$. Then, by lemma \ref{lem1} (\ref{eq18}) is analytically conjugate, on a neighborhood of $(0,0)$, to the vector field 
  \begin{equation}\label{eq21}
Y(\xi,\eta)=g(\xi^2+\eta^2)\bigg(\eta\frac{\partial}{\partial \xi}-\xi\frac{\partial}{\partial \eta}\bigg), \end{equation} where $g$ is an analytic function defined in a neighborhood of the origin such that $g(0)>0.$ By lemma \ref{lem2}, the period function $\Hat{T}(E)$ of (\ref{eq21}) parameterized by $H(\xi,\eta)=E$ is the analytic function given by 
    \begin{equation}
        \Hat{T}(E)=\frac{\pi}{G'(G^{-1}(E))}=\pi\frac{d}{dE}G^{-1}(E),
    \end{equation} where $G$ is the analytic function defined by
\begin{equation}\label{eq23}
    G(z)=\frac{1}{2}\int_0^zg(s)ds.
\end{equation}
By construction, $\Hat{T}(E)=T(E)=\pi\frac{d}{dE}F^{-1}(E)$ and therefore $$\pi\frac{d}{dE}G^{-1}(E)=\pi\frac{d}{dE}F^{-1}(E).$$
Since $G^{-1}(0)=F^{-1}(0)=0$, we have that $$G^{-1}(E)=\int_0^ E\frac{d}{dE}G^{-1}(s)ds=\int_0^E\frac{d}{dE}F^{-1}(s)ds=F^{-1}(E).$$ Therefore $G$ coincides with $F$ in a neighborhood of zero. This implies that the vector fields (\ref{eq6}) and (\ref{eq21}) coincide in a neighborhood of $(0,0)$. Therefore, the vector field $X$ from the definition of theorem (\ref{teo1}) is analytically conjugate, in a neighborhood of $(0,0)$, to the Hamiltonian vector field defined in (\ref{eq18}).\qed

In (\cite{manosas2002area}, Corollary 4.3) it was proved that two Hamiltonian systems with non-degenerate centers at $(0,0)$ are analytically conjugate if and only if both centers have the same period function. The idea of proof the theorem \ref{teo1} was based on this result.

A linear application $R:\R^2\to\R^2$ is called a linear involution over $\mathbb{R}^2$ if $R\neq id$ and $R^2=id$. A vector field $X:\Omega\subseteq\R^2\to\R^2$ is reversible with respect to $R$ or $R$-reversible in $\Omega$ if $R\circ X(q) =-X\circ R(q) \text{ for all }q\in\Omega.$
\begin{corollary}
Let $X=P\frac{\partial}{\partial u}+Q\frac{\partial}{\partial v}$ be an analytic vector field in a neighborhood of $(0,0)$ with $X(0,0)=(0,0)$ and $P_x(0,0) Q_y(0,0)- P_ y(0,0) Q_ x(0,0)>0$. Suppose $X$ is $R$-reversible, with $R(x, y) = (u, -v)$. Then $X$ is analytically conjugate, on a neighborhood of $(0,0)$, to the vector field $Y=y\frac{\partial}{\partial x}-V'(x)\frac{\partial}{\partial y}$, where $V$ is an analytic function defined in a neighborhood of the origin such that $V(0)=V'(0)=0$ and $V''(0)>0.$
\end{corollary}
\begin{proof}
    Under these conditions $X$ has a non-degenerate center at $(0,0)$, see (\cite{teixeira2001center}, Theorem 1).
\end{proof}
 The global equivalence between vector fields of the form $X(u,v)=v\partial_u+f(u,v^2/2)\partial_v$ and the form $Y(x,y)=y\frac{\partial}{\partial x}-V'(x)\frac{\partial}{\partial y}$ was studied in \cite{ragazzo2012scalar} and in the Ph.D. Thesis (in Portuguese) of the author (FJSN). The global version of theorem \ref{teo1} is part of a manuscript by the author in collaboration with C. Grota Ragazzo. 

\medskip

\bibliographystyle{unsrt}
\bibliography{bibliografia}  
\end{document}